\documentclass{amsart}

\usepackage{amsfonts}
\usepackage{amssymb}
\usepackage{amscd}
\usepackage{enumerate}
\usepackage{mathrsfs}

\makeatletter
\renewcommand{\section}{%
\@startsection{section}{1}%
  \z@{.7\linespacing\@plus\linespacing}{.5\linespacing}%
 {\normalfont\large\bfseries\centering}}
\makeatother

\newtheorem{theorem}{Theorem}[section]

\newtheorem{lemma}[theorem]{Lemma}

\newtheorem*{theorem*}{Theorem} 
\newtheorem*{corollary*}{Corollary}
\newtheorem*{conjecture*}{Conjecture}
\newtheorem*{lemma*}{Lemma}
\newtheorem*{proposition*}{Proposition}
\newtheorem*{problem*}{Problem}
\newtheorem*{axiom*}{Axiom}
\newtheorem*{example*}{Example}
\newtheorem*{exercise*}{Exercise}
\newtheorem*{definition*}{Definition}

\theoremstyle{definition}
\newtheorem{remark}{Remark}[section]
\newtheorem*{remark*}{Remark}
\numberwithin{equation}{section} 

\renewcommand{\l}{\left}
\renewcommand{\r}{\right}
\newcommand{\eps}{\varepsilon}
\newcommand{\N}{{\mathbb N}}
\newcommand{\R}{{\mathbb R}}
\newcommand{\C}{{\mathbb C}}

\newcommand{\im}{{\rm Im}}
\newcommand{\re}{{\rm Re}}

\newcommand{\ds}{\displaystyle}

\newcommand{\del}{\partial}

\newcommand{\wto}{\rightharpoonup}

\def\norm[#1]{\left\Vert #1 \right\Vert}
\def\tbra[#1,#2]{\left\langle #1 , #2\right\rangle} 
\def\rbra[#1,#2]{\left( #1 , #2 \right)} 
\def\sbra[#1,#2]{\left[ #1 , #2 \right]} 

\def\besov[#1,#2,#3]{B_{#2,#3}^{#1}}
\def\hbesov[#1,#2,#3]{\dot{B}_{#2,#3}^{#1}}

\newcommand{\scD}{{\mathscr D}}

\newcommand{\scL}{{\mathscr L}}

\begin{document}

\title[NLS in a general domain]{A note on the nonlinear Schr\"{o}dinger equation in a general domain}


\author{Masayuki Hayashi}
\address{Department of Applied Physics, Waseda University, Tokyo 169-8555, Japan}
\curraddr{}
\email{masayuki-884@fuji.waseda.jp}
\thanks{}



\date{}

\dedicatory{}


\begin{abstract}
We consider the Cauchy problem for nonlinear Schr\"{o}dinger equations in a general domain $\Omega\subset\R^N$. Construction of solutions has been only done by classical compactness method in previous results. Here, we construct solutions by a simple alternative approach. More precisely, solutions are constructed by proving that approximate solutions form a Cauchy sequence in some Banach space. We discuss three different types of nonlinearities: power type nonlinearities, logarithmic nonlinearities and damping nonlinearities.
\end{abstract}

\maketitle

\section{Introduction}
We consider the following nonlinear Schr\"{o}dinger equation:
\begin{align}
 \label{NLS}
 \l\{
\begin{array}{lll}
 i \del_{t}u + \Delta u +g(u)= 0  &\text{in} 
&\R \times \Omega , \\[3pt] 
u=0  &\text{on} &\R \times \del \Omega ,\\[3pt]
u(0) =\varphi  &\text{on} &\Omega ,
\end{array}
\r.
\end{align}
where $\Omega \subset \R^N$ is an open set and $g$ is a given nonlinearity. It is often considered the pure power nonlinearity $g(u)=\lambda |u|^{\alpha}u$ with $\lambda \in\R$ and $\alpha\geq 0$ as a typical example.

Strichartz estimates are a key tool for the study of the Cauchy problem (\ref{NLS}). In the case $\Omega =\R^N$, the Cauchy problem for nonlinear Schr\"{o}dinger equations as an application of Strichartz estimates has been extensively studied (see \cite{C03} and references therein).  There is also a large literature for the study of a suitable version of Strichartz estimates with loss on compact manifolds or exterior domains. We refer to \cite{Bo93, ST02, BGT04a, BGT04b, A08, BSS08, I10, BSS12} and references therein for these problems. We note that there are some works in \cite{YT84, T91} for time decay estimates and application to nonlinear problem on exterior domains before these works.  

Strichartz estimates do not hold in the case of general domains.
In this case, compactness method is a useful tool to construct solutions. Vladimirov \cite{V84} considered (\ref{NLS}) on bounded domains in two space dimension  when $g(u)=|u|^2u$ and proved existence and uniqueness of solutions in the energy space $H^1_0(\Omega)$ applying Trudinger's inequality. Ogawa \cite{Og90} improved this result and proved the same results on general domains (not necessarily bounded) when $g(u)=|u|^{\alpha}u$ with $\alpha \leq 2$. They mainly discussed uniqueness of solutions, while construction of solutions was done by Galerkin's method which depends on the pioneer work by J.-L. Lions \cite{L69}. Galerkin's method is a typical method based on compactness arguments. Here, let us review compactness arguments to construct solutions.
\begin{itemize}
\setlength{\itemsep}{3pt}

\item First, we consider approximate problems corresponding to the original equation. Various type of approximations can be considered, e.g., Yosida type regularization, Galerkin approximation, parabolic regularization, truncated approximation, Friedrich's mollifier, etc. It is important to choose a natural regularization according to the nonlinearity. 

\item In a next step, by using the conservation laws  or energy type estimates of approximate equations, one can obtain the uniform estimates on approximate solutions.

\item Under the uniform estimates of approximate solutions, we deduce that there is a limit function of a subsequence of approximate solutions in weak or weak-* topology. In this step, Ascoli-Arzel\`{a}'s theorem or Banach-Alaoglu's theorem can be used.

\item To confirm that the limit function actually satisfies the original equation, we need to strengthen the convergence of approximate solutions, for instance, to prove strong convergence in $L^p(\Omega )$ for some $p \in [1, \infty ]$. In this step, Rellich-Kondrachev's theorem can be used and it is necessary to reduce the convergence problem to arguments on bounded domains in some sense. 

\item After construction of solutions, uniqueness and continuous dependence are discussed separately. It is necessary to start arguments totally different from construction of solutions all over again. 

\end{itemize}

Compactness arguments are applicable to a quite wide type of nonlinear evolution equations as can be seen in \cite{L69}. Our aim in this paper is to construct solutions by a simple alternative approach which is independent of compactness arguments. More precisely, we prove that approximate solutions form a Cauchy sequence in some Banach space. 
Although the range of application to this approach might be restricted compared to compactness methods, there are several advantages in this approach in the following sense.
\begin{itemize}
\setlength{\itemsep}{3pt}

\item As one can see the process above in compactness arguments, we need to use at least two compactness theorems in functional analysis to construct solutions. Our approach of constructing solutions does not need any compactness theorem. 

\item The proof is not only simpler but also more constructive than the one by compactness method since solutions are constructed by using completeness of Banach space directly. 

\item We can establish construction of solutions, uniqueness and continuous dependence by the same procedure. 


\end{itemize}

In this paper, we consider (\ref{NLS}) in the energy space $H^1_0 (\Omega )$. Note that $\Delta$ is self-adjoint in $H^{-1} (\Omega )$ with domain $H^1_0 (\Omega )$, and hence the unitary group $U(t)=e^{it\Delta}$ is well-defined in $H^{-1}(\Omega )$. We consider three different types of nonlinearities as follows; 
\begin{enumerate}[(A)]
\setlength{\itemsep}{5pt}

\item $g(u)$: local nonlinearity which has polynomial growth in 2D,

\item $\ds g(u)=u\log |u|^2$: logarithmic nonlinearity,

\item $\ds g(u)=i\frac{u}{|u|^{\alpha}}$ with $0<\alpha <1$: damping nonlinearity.

\end{enumerate}

For the case (A), there are several previous works which are related to our approach. Cazenave and Haraux \cite[Chapter 7]{CH98} constructed solutions in the energy space $H^1(\R^N)$ by proving that the sequence of truncated approximate solutions is a Cauchy sequence. Note that their argument depends on Strichartz estimates on $\Omega =\R^N$. This approach has an advantage of establishing conservation of energy compared with the contraction argument in Kato's method \cite{K87}. See also \cite{Oz06} for the simple derivation of conservation laws by using integral equations. 
Fujiwara, Machihara and Ozawa \cite{FMO15} considered a system of semirelativistic equations:
\begin{align}
\label{eq:1.2}
 \l\{
\begin{array}{lll}
i \del_t u + (m_u^2-\del_x^2 )^{1/2} u = \lambda \overline{u}v &\text{in} &\R\times\R, \\[3pt] 
i \del_t v + (m_v^2-\del_x^2 )^{1/2} v = \frac{\overline{\lambda} }{c}u^2 &\text{in} &\R\times\R, \\[3pt]
u(0)=\varphi , ~v(0) = \psi  &\text{on} &\R .
\end{array}
\r.
\end{align}
They constructed solutions of (\ref{eq:1.2}) in the energy space $H^{1/2}(\R)\times H^{1/2}(\R)$ by proving that the sequence of approximate solutions is a Cauchy sequence. Their argument is based on a Yudovich argument\footnote{The exponent $p$ of $L^p(\R )$ is used as a parameter in this argument. Since $H^{1/2}(\R) \subset L^p(\R )$ for any $2<p<\infty$, one can take the limit as $p\to \infty$. Combining with the Gronwall inequality, one can prove that approximate solutions form a Cauchy sequence. Originally, this type of argument was used to prove uniqueness in \cite{Y63}. }. We note that some estimates of Yosida type regularization are derived by the technique of Fourier analysis in \cite{FMO15}. 
Spectral analysis would be possible to do similar analysis on general domains, but here we take another simpler approach which depends on using $L^p$-estimate of approximate solutions (see Lemma \ref{lem:2.4}). The author and Ozawa \cite{HO16} considered generalized derivative nonlinear Schr\"{o}dinger equations:
\begin{align}
\label{eq:1.3}
 \l\{
\begin{array}{lll}
 i \del_{t}u + \del_x^2 u +i|u|^{2\sigma}\del_x u= 0 &\text{in} &\R\times\Omega , \\[3pt] 
u=0 &\text{on} &\R\times\del\Omega ,\\[3pt]
 u(0) =\varphi &\text{on} &\Omega ,
\end{array}
\r.
\end{align}
where $\Omega \subset \R$ is an open interval and $\sigma >0$. They constructed $H^2(\Omega )$ solutions of (\ref{eq:1.3}) 
by a similar approach. Since we consider $H^1_0(\Omega )$ solutions in two space dimensions in case (A), we need to consider more careful calculations from the viewpoint of Sobolev embedding. We will consider the problem of case (A) in Section \ref{sec:2}. 

For case (B) and case (C), they are more delicate to apply our approach since nonlinearities have a singularity at the origin. We will discuss these problems in Section \ref{sec:3} and Section \ref{sec:4}.  

\section{The nonlinear Schr\"{o}dinger equation in 2D}
\label{sec:2}
In this section, we consider the equation (\ref{NLS}) in two space dimensions\footnote{In the case of dimension $N=1$, the well-posedness in $H^1_0 (\Omega )$ is proved more easily due to the embedding $H^1_0 (\Omega )\subset L^{\infty} (\Omega )$.}. To state the assumption of nonlinearity $g(u)$, we prepare some notations. The following notation depends on the one in Cazenave \cite[Chapter 3]{C03}.

Let $f: \Omega \times \R_{+} \to \R$ be such that $f(x, u)$ is measurable in $x$ and continuous in $u$. Assume that 
\begin{align*}
f(x, 0) =0 ~\text{for a.e.}~ x \in \Omega
\end{align*}
and that for every $K>0$ there exists $L(K) >0$ such that
\begin{align*}
|f(x, u)-f(x ,v)| \leq L(K) |u-v|~\text{for a.e.}~x \in \Omega ~\text{and}~0\leq u, v \leq K,
\end{align*}
where the function $L:\R_{+}\to\R_{+}$ satisfies
\begin{align*}
L(t) \leq C(1+t^{\alpha}) ~\text{with}~ 0 \leq \alpha < \infty .
\end{align*}
We extend $f$ to the complex plane by
\begin{align*}
f(x,u) = \frac{u}{|u|} f(x, |u|) ~\text{for}~ u \in \C \setminus \{ 0\} ,
\end{align*}
and set 
\begin{align*}
g(u) (x) =f(x,u(x)) ~\text{a.e. $x$ in}~  \Omega 
\end{align*}
for all measurable $u: \Omega \to \C$. We define 
\begin{align*}
F(x, u)=\int_{0}^{u} f(x ,s)ds\quad \text{for}~x\in\Omega ~\text{and}~u\geq 0,
\end{align*}
and
\begin{align*}
G(u) = \int_{\Omega} F(x, |u(x)|) dx.
\end{align*}
The equation (\ref{NLS}) has formally the following conserved quantities:
\begin{align}
\tag{Mass} M(u) &= \int_{\Omega} |u|^2 dx , \\
\tag{Energy} E(u) &= \frac{1}{2}\int_{\Omega} |\nabla u|^2dx - G(u).
\end{align}

Our main aim of this section is to give a simple alternative proof of the following theorem.
\begin{theorem}[Local well-posedness \cite{V84,Og90,C03}]
\label{thm:2.1}
Let $\Omega$ be an open subset of $\R^2$. Let $g$ satisfy the assumption above with $\alpha \leq 2$. For every $\varphi \in H^1_0 (\Omega)$, there exist $0<T_{\rm min}, T_{\rm max}\leq \infty$ and a unique, maximal solution $u\in C((-T_{\rm min},T_{\rm max}) , H^1_0 (\Omega ) )\cap C^1((-T_{\rm min},T_{\rm max}), H^{-1}(\Omega ))$ of the equation \textup{(\ref{NLS})}. Furthermore, the following properties hold:
\begin{enumerate}[\rm (1)]
\setlength{\itemsep}{3pt}

\item $M(u(t))=M(\varphi )$ and $E(u(t))=E(\varphi )$ for all $t \in (-T_{\min}, T_{\max})$.

\item Continuous dependence is satisfied in the following sense; if $\varphi_n\to\varphi$ in $H^1_0(\Omega )$ and if $I\subset (-T_{\rm min} (\varphi ) ,T_{\rm max} (\varphi ) )$ is a closed interval, then the maximal solution $u_n$ of \textup{(\ref{NLS})} with $u_n(0)=\varphi_n$ is defined on $I$ for $n$ large enough and satisfies $u_n \to u$ in $C(I, H^1_0(\Omega ))$.
\end{enumerate}
\end{theorem}
\begin{remark}
When $\alpha >2$, the existence of weak solutions in $H^1_0(\Omega )$ is already known (see \cite[Theorem 3.3.5]{C03}). However, the well-posedness result in $H^1_0(\Omega )$ on general domains in this case seems to be still an open problem. We refer to \cite{BGT04b,A08} for this problem on compact regular domains or exterior domains.
\end{remark}
For the proof of Theorem \ref{thm:2.1}, we consider two types of approximate problems to the equation (\ref{NLS}). One method is to truncate the nonlinearity $g(u)$ for large values of $u$. This seems to be the most appropriate thing for a local nonlinearity. The proof of using truncated approximation is discussed in Section \ref{sec:2.1} and it is related to the arguments in Section \ref{sec:3} and Section \ref{sec:4}. The other method is to use Yosida type regularization by the resolvent of Laplacian. One of the important advantage using this approximation is that one can apply this method to a system of the equations in the same way (see \cite[Remark 3.3.12]{C03} and \cite{FMO15}). It is also possible to treat wider class of the nonlinearities, for instance, which contain derivatives (see \cite{HO16}). For these reasons, we also give the proof of Theorem \ref{thm:2.1} using Yosida type regularization in Section \ref{sec:2.2}.

\subsection{Truncated approximation}
\label{sec:2.1}
Given $m\in\N$. Set
\begin{align*}
g_m(u)&=\l\{
\begin{array}{ll}
g(u)  &~\text{if}~|u| \leq m, \\[5pt]
\ds \frac{u}{m}g(m) &~\text{if}~ |u| \geq m.
\end{array}
\r.
\end{align*}
We consider the following approximate problem:
\begin{align}
\l\{
 \begin{array}{ll}
 i \partial_{t}u_m + \Delta u_m +g_m(u_m)= 0, \\
 u_m(0)=\varphi .
\end{array}
\r. 
\label{eq:2.1}
 \end{align}
Since $g_m:\C\to\C$ is Lipschitz continuous for each $m\in\N$, there exists a unique solution $u_m \in C(\R, H^1_0(\Omega)) \cap C^1(\R ,H^{-1}(\Omega ))$ of (\ref{eq:2.1}). Furthermore, 
\begin{align}
\label{eq:2.2}
M(u_m(t)) =M(\varphi ) \quad \text{and} \quad E_m(u_m (t)) =E_m(\varphi )
\end{align}
for all $t\in \R$. Here, $E_m$ is defined by
\begin{align*}
E_m(u)=\frac{1}{2}\int_{\Omega} |\nabla u|^2dx-G_m(u),
\end{align*}
where $G_m$ is defined by
\begin{align*}
G_m (u) = \int_{\Omega}\int_{0}^{|u(x)|}g_m(s)dsdx
\end{align*}
for all measurable $u:\Omega \to\C$. 
In the same way as  Step 2 in the proof of Theorem 3.3.5 in \cite{C03}, we deduce that there exists $T=T(\varphi ) >0$ such that
\begin{align}
\label{eq:2.3}
M:=\sup_{m \in \N} \| u_m\|_{C([-T,T], H^1_0 )} < \infty .
\end{align}

Next, we prove $(u_m)_{m\in\N}$ is a Cauchy sequence in $C([-T,T], L^2(\Omega ))$. By using the equation (\ref{eq:2.1}), we have
\begin{align*}
\frac{d}{dt}\| u_m-u_n\|_{L^2}^2 
&= 2\mathrm{Im}(i\partial_t u_m-i\partial_t u_n, u_m-u_n)\\
&= -2\mathrm{Im}\rbra[ g_m(u_m)-g_n(u_n),u_m-u_n]\\
&=-2\mathrm{Im}\rbra[ g_m(u_m)-g_n(u_m),u_m-u_n] \\
&\quad -2\mathrm{Im}\rbra[ g_n(u_m)-g_n(u_n),u_m-u_n]\\
&=A_1 +A_2.
\end{align*}

We begin with the estimate of $A_1$. First, we note that 
\begin{align}
\label{eq:2.4}
|g_m(u)-g_m(v)| \leq C(1+|u|^2+|v|^2) |u-v|~\text{for all}~ u , v \in \C ,
\end{align}
where $C$ is independent of $m \in \N$. By (\ref{eq:2.4}), we have
\begin{align*}
|g_m(u_m)-g(u_m)| &\leq\l| \chi_{|u_m| \geq m} \l( \frac{u_m}{m}g(m) -g(u_m) \r)\r| \\
&\leq C\chi_{|u_m| \geq m} |u_m|^3,
\end{align*}
and
\begin{align}
\label{eq:2.5}
\| g_m(u_m)-g(u_m)\|_{L^2}^2 &\leq C \int_{|u_m| \geq m} |u_m|^{6} dx \\
&= C \int_{|u_m| \geq m} |u_m|^{7}|u_m|^{-1} dx \notag\\
&\leq C(M)m^{-1}.\notag
\end{align}
Therefore, $A_1$ is estimated as
\begin{align}
\label{eq:2.6}
A_1 &\leq C\| g_m(u_m) -g_n(u_m)\|_{L^2} \| u_m -u_n\|_{L^2} \\
&\leq 2\| \varphi\|_{L^2}\l( \| g_m(u_m) -g(u_m)\|_{L^2} +\| g(u_m) -g_n(u_m)\|_{L^2} \r) \notag\\
&\leq C(M) \l( \frac{1}{\sqrt{m}} +\frac{1}{\sqrt{n}} \r) .\notag
\end{align}
We estimate $A_2$ by applying a Yudovich argument \cite{Y63}. This is closely related to the argument by Vladimirov \cite{V84} and Ogawa \cite{Og90} (see also \cite{BGT04a}). By (\ref{eq:2.4}), we have
\begin{align*}
A_2 &\leq C\int_{\Omega} (1+|u_m|^2+|u_n|^2) |u_m -u_n|^2 dx \\
&\leq C\| u_m-u_n\|_{L^2}^2 
+C(\| u_m\|_{L^{2p}}^2 +\| u_n\|_{L^{2p}}^2) \| u_m -u_n\|_{L^{2p'}}^2,
\end{align*}
where for any $2<p<\infty$. We use the following estimates
\begin{align*}
\| v\|_{L^{2p'}} &\leq \| v\|_{L^2}^{1-\frac{3}{2p}} \| v\|_{L^6}^{\frac{3}{2p}} ,\\
\| v\|_{L^{2p}} &\leq C \sqrt{p} \| v\|_{H^1},
\end{align*}
where $C$ is independent of $p$. We refer to \cite[Lemma 2]{Og90} for the second inequality. Then, we have
\begin{align}
\label{eq:2.7}
A_2 &\leq \| u_m -u_n\|_{L^2}^2 +C(M)p  \| u_m -u_n\|_{L^2}^{2\l( 1-\frac{3}{2p}\r)} \\
&\leq pC(M) \| u_m -u_n\|_{L^2}^{2\l( 1-\frac{3}{2p}\r)}. \notag
\end{align}
By (\ref{eq:2.6}) and (\ref{eq:2.7}), we obtain
\begin{align*}
\frac{d}{dt} \| u_m-u_n\|_{L^2}^2 \leq 
C(M) \l( \frac{1}{\sqrt{m}} +\frac{1}{\sqrt{n}} \r)
+C(M)p\| u_m -u_n\|_{L^2}^{2\l( 1-\frac{3}{2p}\r)} .
\end{align*}
Applying the Gronwall type inequality, we obtain
\begin{align*}
\| u_m -u_n  \|_{C([-T,T],L^2)}^2 \leq \ \l[ \l( C(M)T \l(  \frac{1}{\sqrt{m}} +\frac{1}{\sqrt{n}} \r) \r)^{\frac{3}{2p}} +C(M)T
\r]^{\frac{2p}{3}}.
\end{align*}
Taking the $\limsup_{m,n\to\infty}$ in the preceding inequality, we have
\begin{align}
\label{eq:2.8}
\limsup_{m,n \to \infty} \| u_m -u_n  \|_{C([-T,T],L^2)}^2 \leq  \l( C(M)T \r)^{\frac{2p}{3}}.
\end{align}
Let $T_0>0$ such that $C(M)T_0 <1$. Taking the limit in (\ref{eq:2.8}) as $p \to \infty$, we deduce that $(u_m)_{m\in\N}$ is a Cauchy sequence in $C([-T_0, T_0], L^2(\Omega ))$. Since $T_0$ only depends on $M$, iterating the same process yields that $(u_m)_{m\in \N}$ is a Cauchy sequence in $C([-T, T], L^2(\Omega ) )$.  Hence, there exists $u \in C([-T,T] , L^2(\Omega ))$ such that $u_m \to u~ \text{in}~C([-T,T], L^2(\Omega ))$. Combining with (\ref{eq:2.2}) and (\ref{eq:2.3}), we deduce that $u \in L^{\infty}([-T,T], H^1_0(\Omega))$ and
\begin{align}
\label{eq:2.9}
M(u(t)) = M(\varphi ) \quad\text{and}\quad E(u(t)) \leq E(\varphi )
\end{align}
for all $t\in [-T,T]$. We also have
\begin{align}
\label{eq:2.10}
u_m \to u~\text{in}~C([-T,T], L^p (\Omega ))
\end{align}
for any $2\leq p<\infty$ by elementary interpolation inequality.

Next, we shall prove that the function $u$ satisfies (\ref{NLS}) and lies in $C([-T,T], H^1_0(\Omega ))$. Note that $u_m$ is a solution of the integral equation
\begin{align}
\label{eq:2.11}
u_m(t)=U(t)\varphi +i\int_{0}^{t} U(t-s)g_m(u_m(s))ds \quad\text{for all}~t\in [-T,T]. 
\end{align}
We write 
\begin{align*}
|g_m(u_m)-g(u)|\leq |g_m(u_m)-g(u_m) | +|g(u_m) -g(u)|.
\end{align*}
By (\ref{eq:2.5}), we have
\begin{align*}
\| g_m(u_m)-g(u_m)\|_{C([-T,T],L^2)} \leq C(M)m^{-\frac{1}{2}}
\underset{m\to\infty}{\longrightarrow} 0.
\end{align*}
By (\ref{eq:2.4}) and (\ref{eq:2.10}), it is easily verified that
\begin{align*}
\| g(u_m)-g(u)\|_{C([-T,T],L^2)}\underset{m\to\infty}{\longrightarrow} 0.
\end{align*}
Hence, we deduce that $g_m(u_m)\to g(u)$ in $C([-T,T], L^2(\Omega ))$. Taking the limit in the integral equation (\ref{eq:2.11}) as $m \to \infty$, we conclude that
\begin{align}
\label{eq:2.12}
u(t) = U(t)\varphi +i\int_{0}^{t} U(t-s)g(u(s)) ds\quad\text{for all}~t\in [-T,T].  
\end{align}
Since $g(u) \in C([-T,T], H^{-1}(\Omega))$, it follows that $u \in C^1([-T,T],H^{-1}(\Omega))$. 
Therefore, $u$ satisfies the equation (\ref{NLS}). It follows from the equation (\ref{NLS}) that $\Delta u \in C([-T,T], H^{-1}(\Omega))$. Hence, we deduce that $u \in C([-T,T], H^1_0(\Omega ))$. 

If the uniqueness holds, it follows from (\ref{eq:2.9}) that $E(u(t))=E(\varphi )$ for all $t \in [-T,T]$. Uniqueness and continuous dependence are proved by the same process as the construction of solution above. We only prove continuous dependence. 

Let $(-T_{\min}, T_{\max})$ be the maximal interval of existence to the solution $u$. Let $\varphi_n \to \varphi$ in $H^1_0(\Omega )$ and closed interval $I \subset (-T_{\min}, T_{\max})$ and consider the maximal solution $u_n$ with initial condition $u_n (0)=\varphi_n$.  From the conservation laws, we obtain that
\begin{align*}
M=\sup_{n \in \N} \| u_n\|_{C(I, H^1_0 )} \leq C(\| \varphi \|_{H^1_0}) < \infty.
\end{align*}
By using the equation (\ref{NLS}), we have
\begin{align*}
\frac{d}{dt}\| u-u_n\|_{L^2}^2 &= -2\mathrm{Im}\rbra[ g(u)-g(u_n),u-u_n]\\
&\leq \| u -u_n\|_{L^2}^2 +C(M)p  \| u -u_n\|_{L^2}^{2\l( 1-\frac{3}{2p}\r)} \\
&\leq pC(M) \| u -u_n\|_{L^2}^{2\l( 1-\frac{3}{2p}\r)} ,
\end{align*}
where we estimated in the same way as $A_2$. Applying the Gronwall type inequality, we deduce that
\begin{align*}
\| u -u_n \|_{C([-T,T],L^2)}^2 \leq \ \l( C(M)T\r)^{\frac{2p}{3}}.
\end{align*}
If we take $T_0>0$ such that $C(M)T_0 <1$ and $p \to \infty$, we deduce that $u_n \to u$ in $C([-T_0, T_0], L^2(\Omega ))$. Iterating the same process yields that $u_n \to u$ in $C(I, L^2(\Omega) )$. Conservation of energy implies that $\| \nabla u_n\|_{L^2}^2$ converges to $\| \nabla u\|_{L^2}^2$ uniformly on $I$. Hence, we deduce that  $u_n \to u$ in $C(I, H^1_0(\Omega ))$ and this completes the proof.
%
\subsection{Yosida type regularization}
\label{sec:2.2}
Given $m\in\N$, we define the Yosida type operator $J_m$ on $H^{-1}(\Omega)$ by
\begin{align}
\label{eq:2.13}
J_m=\left( I-\frac{1}{m}\Delta \right )^{-1}.
\end{align}
More precisely, for every $f \in H^{-1}(\Omega)$, $v_m=J_mf \in H^1_0 (\Omega )$ is the unique solution of 
\begin{align*}
v_m -\frac{1}{m}\Delta v_m =f.
\end{align*}
We recall the fundamental properties of the operator $J_m$. We refer to \cite{Br11, CH98, C03} for the proof of following lemmas.  
\begin{lemma}
\label{lem:2.2}
Let $X$ be any of the spaces $H^1_0 (\Omega ),~ L^2(\Omega)$ or $H^{-1}(\Omega)$. Then, 
\begin{align*}
\| J_mf\|_X \leq \| f \|_X \quad \text{for all}~f \in X.
\end{align*}
If $f \in L^p (\Omega) \cap H^{-1}(\Omega )$ for some $p\in [1,\infty)$, then $J_mf \in L^p(\Omega )$ and
\begin{align*}
\| J_mf\|_{L^p} \leq \| f\|_{L^p}.
\end{align*}
In particular, $J_m$ can be uniquely extended by continuity to an operator of $\scL (L^p(\Omega ))$ with $\| J_m\|_{\scL (L^p)} \leq 1$, where $1\leq p<\infty$.
\end{lemma}
\begin{lemma}
\label{lem:2.3}
Let $X$ be any of the spaces $H^1_0(\Omega),\ H^{-1}(\Omega),\ $and $L^p(\Omega)$ with $1<p<\infty$ and let $X^{\ast}$ be its dual space.  Then,
\begin{enumerate}[\rm (1)]
\setlength{\itemsep}{3pt}

\item $\tbra[J_mf,g]_{X,X^{\ast}}=\tbra[f,J_mg]_{X,X^{\ast}}$ for all $f\in X, g\in X^{\ast}$.

\item $J_m\in \scL (H^{-1},H^1_0)$ and $\| J_m\|_{\scL (H^{-1}, H^1_0)} \leq m$.

\item  $J_mu \rightarrow u $ in $X$ as $m\rightarrow \infty$ for every $u \in X$.

\end{enumerate}
\end{lemma}
We consider the following approximate problem:
\begin{align}
 \begin{cases}
 i \partial_{t}u_m + \Delta u_m +J_mg(J_mu_m)= 0, \\
 u_m(0)=\varphi .
 \end{cases}
 \label{eq:2.14}
 \end{align}
For each $m\in\N$, there exists a unique solution $u_m \in C(\R, H^1_0(\Omega)) \cap C^1(\R ,H^{-1}(\Omega ))$ of (\ref{eq:2.14}) and
\begin{align}
\label{eq:2.15}
M(u_m(t)) =M(\varphi ) \quad \text{and} \quad E_m(u_m (t)) =E_m(\varphi )
\end{align}
for all $t\in \R$. Here, $E_m$ is defined by 
\begin{align*}
E_m(u)=\frac{1}{2}\int_{\Omega} |\nabla u|^2dx -G_m(u),
\end{align*}
where $G_m$ is defined by $G_m(u)=G(J_mu)$ for every $u\in H^1_0 (\Omega )$.
From the conservation laws (\ref{eq:2.15}), there exists $T=T(\varphi )>0$ such that
\begin{align}
\label{eq:2.16}
M:=\sup_{m \in \N} \| u_m\|_{C([-T, T], H^1_0 )} < \infty .
\end{align}

Next, we prove $(u_m)_{m\in\N}$ is a Cauchy sequence in $C([-T,T], L^2(\Omega ))$. By using the equation (\ref{eq:2.14}), we have
\begin{align*}
\frac{d}{dt}\| u_m-u_n\|_{L^2}^2 &= 2\mathrm{Im}(i\partial_t u_m-i\partial_t u_n, u_m-u_n)\\
&=-2\mathrm{Im}(J_mg(J_mu_m)-J_ng(J_n u_n),u_m-u_n)\\
&=-2\mathrm{Im} \Bigl[ \rbra[ J_mg(J_mu_m)-J_ng(J_mu_m),u_m-u_n] \\
&\quad + \rbra[ g(J_mu_m) -g(J_nu_m) , J_n(u_m-u_n)]\\
&\quad + \rbra[ g(J_nu_m) -g(J_nu_n) , J_n(u_m-u_n)] \Bigr]\\
&=B_1+B_2+B_3.
\end{align*}
In the same way as $A_2$ in Section \ref{sec:2.1}, $B_3$ is estimated by
\begin{align}
\label{eq:2.17}
B_3 &\leq pC(M) \| u_m -u_n\|_{L^2}^{2\l( 1-\frac{3}{2p}\r)} .
\end{align} 
In order to estimate $B_1$ and $B_2$, we prepare the following key lemma.
\begin{lemma}
\label{lem:2.4}
Let $1<p<\infty$. Then,
\begin{align}
\label{eq:2.18}
\| J_m v - J_n v\|_{L^p} \leq \sqrt{2}\l( \frac{1}{\sqrt{m}} +\frac{1}{\sqrt{n}}\r) \| v\|_{W^{1,p}} 
\end{align}
for all $v \in W^{1,p}_0 (\Omega )$.
\end{lemma}
\begin{proof}
Let the operator $T_m =\Delta J_m$ and $v \in C^{\infty}_c (\Omega ) $. By Lemma \ref{lem:2.2}, we have
\begin{align}
\label{eq:2.19}
\| T_m v\|_{L^p} =\| J_m\Delta v\|_{L^p}\leq \| \Delta v \|_{L^p} \leq \| v\|_{W^{2,p}}.
\end{align}
From the definition (\ref{eq:2.13}) of $J_m$, we obtain
\begin{align*}
\Delta J_m v =m (J_m v-v).
\end{align*}
By Lemma \ref{lem:2.2} again, we have
\begin{align}
\label{eq:2.20}
\| T_m v\|_{L^p} \leq m (\| J_m v\|_{L^p} +\| v\|_{L^p} ) \leq 2m \| v\|_{L^p}.
\end{align}
By (\ref{eq:2.19}), (\ref{eq:2.20}) and the complex interpolation\footnote{Since $v\in C^{\infty}_c (\Omega )$, if we extend $v$ by zero outside $\Omega$, we can use the interpolation on the whole space.}(see \cite{AF03, BL76}), 
we deduce that
\begin{align}
\label{eq:2.21}
\| T_m v\|_{L^p} \leq \sqrt{2m} \| v\|_{W^{1,p}}.
\end{align}
We write
\begin{align*}
J_m v- J_n v &=(J_m -1) v+(1- J_n) v \\
&=\frac{1}{m}\Delta J_m v - \frac{1}{n} \Delta J_n v .
\end{align*}
Applying (\ref{eq:2.21}), we deduce that
\begin{align*}
\| J_mv - J_n v\|_{L^p} &\leq \frac{1}{m}\| T_m v\|_{L^p} +\frac{1}{n} \| T_n v\|_{L^p} \\
&\leq \sqrt{2} \l( \frac{1}{\sqrt{m}}+\frac{1}{\sqrt{n}}\r) \| v\|_{W^{1,p}} .
\end{align*}
The result follows from a density argument.
\end{proof}
Fix a function such that $\theta \in C^{\infty}_c (\C , \R)$ is radial and
\begin{align*}
\theta (x)=\l\{
\begin{array}{ll}
1  &~\text{if}~ |x| \leq 1,\\[3pt]
0  &~\text{if}~ |x| \geq 2.
\end{array}
\r.
\end{align*}
Set
\begin{align*}
g_1 (u) &= \theta (u) g(u) ,\\
g_2(u) &= (1-\theta (u) )g(u).
\end{align*}
By (\ref{eq:2.4}), we have
\begin{align}
\label{eq:2.22}
\begin{split}
| g_1(u) -g_1 (v) | &\leq C|u-v| ,\\
| g_2(u) - g_2(v) | &\leq C(|u|^2+|v|^2) |u-v|.
\end{split}
\end{align}
One can easily verify that
\begin{align}
\label{eq:2.23}
\begin{split}
\| g_1 (u)\|_{H^1} &\leq C \| u\|_{H^1}, \\
\| g_2(u)\|_{W^{1,r}} &\leq C \| u\|_{L^q}^2 \| u\|_{H^1},
\end{split}
\end{align}
where $1<r<2$, $4<q<\infty$ and $(r,q)$ satisfies $\frac{1}{r} =\frac{2}{q} +\frac{1}{2}$. In particular, we note that $g_1 (u) \in H^1_0 (\Omega)$ and $g_2(u) \in W^{1,r}_0 (\Omega )$ if $u \in H^1_0(\Omega)$. Fix the pair $(r,q)$. We write
\begin{align*}
B_1&=-2\mathrm{Im} \rbra[ J_mg(J_mu_m)-J_ng(J_mu_m),u_m-u_n] \\
&=-2\mathrm{Im} \rbra[ J_mg_1(J_mu_m)-J_ng_1(J_mu_m),u_m-u_n] \\
&\quad -2\mathrm{Im} \rbra[ J_mg_2(J_mu_m)-J_ng_2(J_mu_m),u_m-u_n] \\
&= B_{11}+B_{12}.
\end{align*}
By H\"older's inequality, Lemma~\ref{lem:2.4} and (\ref{eq:2.23}), we obtain that
\begin{align}
\label{eq:2.24}
B_{11} &\leq 2 \| (J_m-J_n)g_1(J_mu_m)\|_{L^2} \| u_m -u_n\|_{L^2} \\
&\leq C(M) \l( \frac{1}{\sqrt{m}} +\frac{1}{\sqrt{n}}\r) \| g_1(J_mu_m) \|_{H^1} \notag\\
&\leq C(M) \l( \frac{1}{\sqrt{m}} +\frac{1}{\sqrt{n}} \r) \notag
\end{align}
and
\begin{align}
\label{eq:2.25}
B_{12} &\leq 2 \| (J_m-J_n)g_2(J_mu_m)\|_{L^r} \| u_m -u_n\|_{L^{r'}} \\
&\leq C(M) \l( \frac{1}{\sqrt{m}} +\frac{1}{\sqrt{n}}\r) \| g_2(J_mu_m) \|_{W^{1,r}} \notag\\
&\leq C(M) \l( \frac{1}{\sqrt{m}} +\frac{1}{\sqrt{n}} \r) . \notag
\end{align} 
$B_2$ is estimated similarly. We write
\begin{align*}
 B_2&=-2\mathrm{Im}\rbra[ g(J_mu_m) -g(J_nu_m) , J_n(u_m-u_n)] \\
 &=-2\mathrm{Im}\rbra[ g_1(J_mu_m) -g_1(J_nu_m) , J_n(u_m-u_n)]\\
&\quad -2\mathrm{Im}\rbra[ g_2(J_mu_m) -g_2(J_nu_m) , J_n(u_m-u_n)]   \\
&=B_{21}+B_{22}.
\end{align*}
By H\"older's inequality, (\ref{eq:2.22}) and Lemma~\ref{lem:2.4}, we obtain that
\begin{align}
\label{eq:2.26}
B_{21} &\leq 2 \| g_1(J_mu_m)-g_1(J_nu_m)\|_{L^2} \| J_n(u_m -u_n)\|_{L^2} \\
&\leq C(M) \| J_mu_m -J_n u_m\|_{L^2} \notag\\
&\leq C(M) \l( \frac{1}{\sqrt{m}} +\frac{1}{\sqrt{n}} \r) \notag
\end{align}
and
\begin{align}
\label{eq:2.27}
B_{22} &\leq 2 \| g_2(J_mu_m)-g_2(J_nu_m)\|_{L^r} \| J_n(u_m -u_n)\|_{L^{r'}} \\
&\leq C(M) \l( \| J_mu_m\|_{L^q}^2+\| J_nu_m\|_{L^q}^2 \r) \| J_mu_m -J_n u_m\|_{L^2} \| u_m -u_n\|_{L^{r'}} \notag\\
&\leq C(M) \l( \frac{1}{\sqrt{m}} +\frac{1}{\sqrt{n}} \r) .\notag
\end{align} 
By (\ref{eq:2.24}), (\ref{eq:2.25}), (\ref{eq:2.26}) and (\ref{eq:2.27}), we deduce that
\begin{align*}
\frac{d}{dt} \| u_m-u_n\|_{L^2}^2 \leq C(M) \l( \frac{1}{\sqrt{m}} +\frac{1}{\sqrt{n}} \r)
+C(M)p\| u_m -u_n\|_{L^2}^{2\l( 1-\frac{3}{2p}\r)} .
\end{align*}
This yields that $(u_m)_{m\in \N}$ is a Cauchy sequence in $C([-T,T], L^2(\Omega ))$. The rest of the proof is done in the similar way as Section \ref{sec:2.1}. We omit the detail.


\section{The logarithmic Schr\"{o}dinger equation}
\label{sec:3}
In this section, we consider the following  logarithmic Schr\"{o}dinger equation:
\begin{align}
 \label{LSE}
 \l\{
\begin{array}{lll}
\ds i \del_{t}u + \Delta u +u\log (|u|^2)= 0 &\text{in} &\R\times\Omega , \\[3pt]
\ds u=0  &\text{on}  &\R \times \del \Omega ,\\[3pt]
u(0) =\varphi   &\text{on} &\Omega ,
\end{array}
\r.
\end{align}
where $\Omega \subset\R^N$ is an open set. This equation appears in various fields of physics: quantum mechanics, quantum optics, nuclear physics, Bohmian mechanics, etc.; see \cite{BM75, H85, Z10, NM17}. Cazenave and Haraux \cite{CH80} studied the Cauchy problem for the equation (\ref{LSE}) (see also \cite{GLN10, CG16}). Cazenave \cite{C83} studied stable solutions for (\ref{LSE}). We also refer to \cite{dMS14, SS15, TZ17} and references therein for recent related works on stable solutions. 
%

To state our main result, we need to prepare some natations. We define the functions $F,A,B,a,b$ on $\R_{+}$ by
\begin{align*}
F(x)&= x^2\log x^2, \\
A(x)&=
\l\{
\begin{array}{ll}
\ds -x^2\log x^2 &~\text{if}~0< x\leq e^{-3},\\
\ds 3x^2+4e^{-3}x-e^{-6} &~\text{if}~ x\geq e^{-3},
\end{array}
\r. 
\\
B(x)&=F(x)+A(x),
\end{align*}
and
\begin{align*}
a(x) =\frac{A(x)}{x}, \quad b(x)=\frac{B(x)}{x}.
\end{align*}
We extend the function $a$ and $b$ to the complex plane by setting
\begin{align*}
a(z) =\frac{z}{|z|}a(|z|), \quad b(z)=\frac{z}{|z|}b(|z|)\quad \text{for}~z \in \C\setminus\{0\} . 
\end{align*}
Let $A^*$ be the convex conjugate function of $A$. Define the sets $X$ and $X'$ by
\begin{align*}
X=\l\{ u\in L^1_{\rm loc}(\Omega )~;~A(|u|) \in L^1 (\Omega )\r\},\quad
X'=\l\{ u\in L^1_{\rm loc}(\Omega )~;~A^{*}(|u|) \in L^1 (\Omega )\r\}.
\end{align*}
We use the Luxemburg norm defined by
\begin{align*}
\| u\|_X &=\inf \l\{ k>0~;~\int_{\Omega} A\l( \frac{|u|}{k}\r) \leq 1\r\}~\text{for}~u \in X, \\
\| u\|_{X'} &=\inf \l\{ k>0~;~\int_{\Omega} A^*\l( \frac{|u|}{k}\r) \leq 1\r\}~\text{for}~u \in X'.
\end{align*}
$X$ and $X'$ are called Orlicz spaces. We refer to \cite{AF03, KR61} for Orlicz spaces in detail. 
First, we prepare the following basic lemma.
\begin{lemma}[\cite{C83}]
\label{lem:3.1}
$(X, \| \cdot\|_X)$ and $(X', \| \cdot\|_{X'})$ are reflexive Banach spaces and $X'$ is the topological dual of $X$. Furthermore, the following properties hold:
\begin{enumerate}[\rm (1)]
 \setlength{\itemsep}{3pt}

\item If $u_m \underset{m\to\infty}{\longrightarrow} u$ in $X$, then $A(|u_m|) \underset{m\to\infty}{\longrightarrow} A(|u|)$ in $L^1(\Omega )$. 

\item If $u_m \underset{m\to\infty}{\longrightarrow}u$ a.e. and if
\begin{align*}
\int_{\Omega} A(|u_m|) \underset{m\to\infty}{\longrightarrow} \int_{\Omega} A(|u|) <\infty ,
\end{align*}
then $u_m \underset{m\to\infty}{\longrightarrow} u$ in $X$.
\end{enumerate} 
\end{lemma}
Set the Banach space
\begin{align*}
W=H^1_0 (\Omega ) \cap X .
\end{align*}
Note that the dual space of $W$ is given by
\begin{align*}
W^{\ast} = H^{-1}(\Omega ) +X' .
\end{align*}
Define the energy functional
\begin{align*}
E(u)=\frac{1}{2}\int_{\Omega} |\nabla u|^2 -\frac{1}{2}\int_{\Omega}|u|^2\log |u|^2
\end{align*}
for $u\in W$. We have the following result.
\begin{lemma}[\cite{C83}]
\label{lem:3.2}
The operator $L:u\mapsto \Delta u+u\log |u|^2$ maps continuously $W \to W^{\ast}$. The image under $L$ of a bounded subset of $W$ is a bounded subset of $W^{\ast}$. The operator $E$ belongs to $C^1 (\R ,W)$ and 
\begin{align*}
E'(u) =-Lu-u
\end{align*}
for all $u \in W$.
\end{lemma}
Our main aim of this section is to give a simple alternative proof of the following theorem.
\begin{theorem}[\cite{CH80, C03}]
\label{thm:3.3}
Let $N\geq 1$. For every $\varphi \in W$, there exists a unique solution $u \in C(\R, W)\cap C^1(\R, W^{\ast})$ of the equation \textup{(\ref{LSE})}. Furthermore, the following properties hold:
\begin{enumerate}[\rm (1)]
\setlength{\itemsep}{3pt}

\item $\sup_{t\in\R}\| u(t)\|_W < \infty$.

\item $M(u(t))=M(\varphi )$ and $E(u(t))=E(\varphi )$ for all $t\in\R$.

\item Continuous dependence is satisfied in the sense that if $\varphi_m\to\varphi$ in $W$, then $u_m\to u$ in $W$ uniformly on bounded intervals, where $u_m$ is the solution of \textup{(\ref{LSE})} with $u_m(0)=\varphi_m$.

\end{enumerate}
\end{theorem}
For the proof of Theorem \ref{thm:3.3}, the following lemma is important in our analysis.
\begin{lemma}[\cite{CH80}]
\label{lem:3.4}
For all $u, v \in \C$, we have
\begin{align*}
\l| \im \l[ (u\log |u|^2-v\log |v|^2)(\overline{u} -\overline{v})\r] \r| \leq 2|u-v|^2.
\end{align*}
\end{lemma}
\begin{proof}
For the convenience of the readers, we prove this lemma. We write
\begin{align*}
\im \l[ (u\log |u|^2-v\log |v|^2)(\overline{u} -\overline{v})\r] &=\im \l[-u\log |u|^2\overline{v}-v\log |v|^2\overline{u}\r]
\\
&=\l( \log |u|^2-\log |v|^2\r) \im (\overline{u}v)
\end{align*}
It is easily verified that
\begin{align*}
\im (u\overline{v}) \leq \min \{ |u|, |v|\} |u-v|.
\end{align*}
On the other hand, we have
\begin{align*}
\l| \log |u|-\log |v|\r| &=\l| \int_{0}^{1}\frac{d}{dt}\log (t|u|+(1-t)|v|) dt\r|\\
&=\l|\int_{0}^{1} \frac{1}{t|u|+(1-t)|v|} (|u|-|v|) dt \r|\\
&\leq \frac{1}{\min\{|u|, |v|\}} |u-v|.
\end{align*}
Therefore, we have
\begin{align*}
\l| \l( \log |u|^2-\log |v|^2\r) \im (\overline{u}v)\r| \leq 2|u-v|^2.
\end{align*}
This completes the proof.
\end{proof}
We consider the approximate equation of (\ref{LSE}). Given $m \in \N$, and define the functions $a_m$ and $b_m$ by
\begin{align*}
a_m(z)=
\l\{
\begin{array}{ll}
a(z)  &~\text{if}~ |z| \geq \frac{1}{m},\\[3pt]
mza\l( \frac{1}{m}\r)  &~\text{if}~  |z| \leq \frac{1}{m},
\end{array}
\r.
b_m(z)=\l\{
\begin{array}{ll}
 b(z)  &\ds ~\text{if}~|z| \leq m,\\[3pt]
\frac{z}{m}b\l( m\r)  & \ds ~\text{if}~|z| \geq m.
\end{array}
\r.
\end{align*}
Set
\begin{align*}
g_m(u) =-a_m(u)+b_m(u)
\end{align*}
and consider the following problem:
\begin{align}
\label{eq:3.2}
\l\{
\begin{array}{ll}
i\del_t u_m +\Delta u_m +g_m (u_m) =0,\\
u_m(0) =\varphi .
\end{array}
\r.
\end{align}
It is easily verified that there exists a unique solution $u_m \in C(\R, H^1_0 (\Omega ))\cap C^1 (\R ,H^{-1}(\Omega ))$ of (\ref{eq:3.2}) and
\begin{align}
\label{eq:3.3}
M(u_m(t)) =M(\varphi ) \quad \text{and} \quad E_m(u_m (t)) =E_m(\varphi )
\end{align}
for all $t\in \R$. Here, $E_m$ is defined by
\begin{align*}
E_m(u)=\frac{1}{2}\int_{\Omega} |\nabla u|^2+\int_{\Omega} \Phi_m (|u|) -\int_{\Omega}\Psi_m (|u|), 
\end{align*}
where $\Phi_m$ and $\Psi_m$ are defined by
\begin{align*}
\Phi_m (z) = \int_{0}^{|z|}a_m(s)ds ,\quad \Psi_m (z) =\int_{0}^{|z|} b_m (s)ds
\end{align*}
for all $z \in \C$. For $z>0$, we have
\begin{align}
\label{eq:3.4}
\int_{0}^{z}s\log s^2ds =\frac{z^2}{2}\log z^2-\frac{1}{2}z^2.
\end{align}
By the dominated convergence theorem and (\ref{eq:3.4}),
\begin{align}
\label{eq:3.5}
E_m(\varphi ) \underset{m\to\infty}{\longrightarrow} E(\varphi )+\frac{1}{2}M(\varphi ).
\end{align}
By (\ref{eq:3.3}), (\ref{eq:3.5}) and $\Phi_m \geq 0$, we obtain
\begin{align}
\label{eq:3.6}
\| \nabla u_m(t) \|_{L^2}^2 \leq C(\| \varphi \|_W ) < \infty 
\end{align}
for all $t \in \R$. By (\ref{eq:3.3}), (\ref{eq:3.5}) and (\ref{eq:3.6}), we deduce that
\begin{align}
\label{eq:3.7}
\int_{\Omega}\Phi_m (|u_m(t)|) &=E_m (\varphi )-\frac{1}{2}\| \nabla u_m(t)\|_{L^2}^2-\Psi_m (|u_m(t)|) \\
&\leq C(\| \varphi \|_W ) \notag
\end{align}
for all $t\in \R$. We set 
\begin{align}
\label{eq:3.8}
M := \sup_{m \in\N} \| u_m\|_{L^{\infty}(\R ,H^1_0)}<\infty .
\end{align}

Next, we prove $(u_m)_{m\in\N}$ is a Cauchy sequence in $C([-T,T] , L^2(\Omega ' ))$ for all $T>0$ and $\Omega' \subset\subset \Omega$. 
We fix a function $\psi \in C^{\infty}_c (\R^N)$ such that $\psi$ is radial and
\begin{align*}
\psi (x)=
\l\{
\begin{array}{ll}
1 & ~\text{if}~|x|\leq 1, \\[3pt]
0& ~\text{if}~|x|\geq 2,
\end{array}
\r.
\end{align*}
and set $\psi_R (x) =\psi (x/R)$. By using the equation (\ref{eq:3.2}) and (\ref{eq:3.8}), we have
 \begin{align*}
 \frac{d}{dt}\| \psi_R (u_m-u_n)\|_{L^2}^2 &= 2\mathrm{Im}\rbra[\psi_R^2 (i\partial_t u_m-i\partial_t u_n), u_m-u_n]\\
 &=2\im \rbra[ \nabla (\psi_R^2 )\cdot\nabla (u_m-u_n),u_m-u_n]\\
&\quad -2\im \rbra[\psi_R^2 (g_m(u_m)-g_n(u_n) ), u_m-u_n] \\
&\leq \frac{C(M)}{R} +L_1+L_2 ,
 \end{align*}
 where $L_1$ and $L_2$ are defined by
\begin{align*}
L_1 &=  2\im \rbra[\psi_R^2 (a_m(u_m)-a_n(u_n) ), u_m-u_n] , \\
L_2&= -2\im \rbra[\psi_R^2 (b_m(u_m)-b_n(u_n) ), u_m-u_n] .
\end{align*}
To estimate $L_1$ and $L_2$, we prove the following key lemma.
 \begin{lemma}
 \label{lem:3.5}
 Let $n \in \N$. For all $u, v \in \C$, we have 
 \begin{align}
 \label{eq:3.9}
 \l| \im \l[ (a_n(u)-a_n(v) )(\overline{u} -\overline{v})\r]\r| \leq C|u-v|^2, 
\\ 
\label{eq:3.10}
|\im \l[ (b_n(u)-b_n(v) )(\overline{u} -\overline{v})\r]| \leq C|u-v|^2,
 \end{align}
 where $C$ is independent of $n$.
 \end{lemma}
 \begin{proof}
 We only prove (\ref{eq:3.9}). One can prove (\ref{eq:3.10}) in a similar way. First, we recall that 
 \begin{align*}
 a(x) =\l\{
\begin{array}{ll}
-x\log x^2 &~\text{if}~ 0< x\leq e^{-3},\\[3pt]
\ds 3x+4e^{-3}-\frac{e^{-6}}{x} &~\text{if}~x\geq e^{-3}.
\end{array}
\r. 
 \end{align*}
 We may consider only the case $\frac{1}{n}<e^{-3}$. Define
 \begin{align*}
 &R_1 =\l\{ z\in \C ~;~0<|z|<\frac{1}{n}\r\},\\
 &R_2 =\l\{ z\in \C ~;~\frac{1}{n}\leq |z|\leq e^{-3}\r\},\\
 &R_3 =\l\{ z\in \C ~;~e^{-3}<|z| \r\},
 \end{align*}
 and
 \begin{align*}
 I_n= \im \l[ (a_n(u)-a_n(v) )(\overline{u} -\overline{v})\r] .
 \end{align*}
 When $u, v \in R_1$, $I_n=0$. When $u,v \in R_2$, we have 
\begin{align*}
|I_n| \leq 2|u-v|^2
\end{align*}
by Lemma~\ref{lem:3.4}. When $u, v \in R_3$, a simple calculation shows that (\ref{eq:3.9}) holds. 

Next, we consider that each of $u$ and $v$ belongs to the different region. Without loss of generality, we may assume $0<|v|<|u|$. Note that
\begin{align*}
I_n &=\im \l[ -a_n(u)\overline{v}-a_n(v)\overline{u}\r] \\
&= \im \l[ -\frac{u}{|u|}a_n(|u|)\overline{v} -\frac{v}{|v|}a_n(|v|)\overline{u}\r]\\
&=\l[ \frac{a_n(|u|)}{|u|}-\frac{a_n(|v|)}{|v|}\r] \im (\overline{u}v).
\end{align*}
Set $ h_n(s) =\frac{a_n(s)}{s}$ for $s>0$. We write
\begin{align}
\label{eq:3.11}
I_n =(h_n(|u|)-h_n(|v|) ) \im (\overline{u}v).
\end{align}
When $v\in R_1$ and $u \in R_2$, we take $w \in \C$ such that
\begin{align*}
|w|=\frac{1}{n}\quad\text{and}\quad |u-v|=|u-w|+|w-v|.
\end{align*}
We note that $h_n(|v|)=-\log(\frac{1}{n})^2=h_n(|w|)$ and
\begin{align*}
\l| h_n(|u|)-h_n(|w|)\r| &= \l| -\log |u|^2 +\log |w|^2 \r|\\
&\leq \frac{2}{|w|}|u-w|.
\end{align*}
Since
\begin{align}
\label{eq:3.12}
\l| \im (\overline{u}v)\r| \leq |v|\,|u-v|,
\end{align}
we deduce that
\begin{align*}
|I_n| &\leq \l( \l| h_n(|u|)-h_n(|w|) \r| +\l| h_n(|w|)-h_n (|v|)\r| \r) \l| \im (\overline{u}v)\r| \\
&\leq \frac{2}{|w|}|u-w| \cdot |v|\, |u-v| \\
&\leq 2|u-v|^2.
\end{align*}
Similarly, one can prove that (\ref{eq:3.9}) holds when $v \in R_2$ and $u \in R_3$. 

Finally, we consider the case $v\in R_1$ and $u \in R_3$. We take $w_1, w_2 \in \C$ such that
\begin{align*}
|w_1|=e^{-3}, |w_2|=\frac{1}{n}\quad\text{and}\quad |u-v|=|u-w_1|+|w_1-w_2|+|w_2-v|.
\end{align*}
A simple calculation shows that 
\begin{align}
\label{eq:3.13}
\begin{array}{l}
\ds \l| h_n(|u|)-h_n(|w_1|)\r| \leq C|u-w_1|, \\[3pt]
\ds \l| h_n(|w_1|)-h_n(|w_2|)\r| \leq \frac{2}{|w_2|}|w_1-w_2|, \\[3pt]
\ds \l| h_n(|w_2|)-h_n(|v|)\r| =0,
\end{array}
\end{align}
where $C$ is independent of $n$. By (\ref{eq:3.12}) and (\ref{eq:3.13}), we deduce that
\begin{align*}
|I_n| &\leq \l( C|u-w_1|+\frac{2}{|w_2|}|w_1-w_2|\r) |v|\,|u-v|\\
&\leq C\l( |u-w_1|+|w_1-w_2|\r) |u-v| \\
&\leq C|u-v|^2.
\end{align*}
This completes the proof.
 \end{proof}
We begin with the estimate of $L_1$. We write
 \begin{align*}
 L_1 &= 2\im \rbra[\psi_R^2 (a_m(u_m)-a_n(u_m) ), u_m-u_n]\\
 &\quad +2\im \rbra[\psi_R^2 (a_n(u_m)-a_n(u_n) ), u_m-u_n]\\
 &=L_{11}+L_{12}.
 \end{align*}
 We note that
 \begin{align*}
 \l| a_m(u_m)-a(u_m)\r| &=\l| \chi_{|u_m| \leq \frac{1}{m}} \l( mu_ma\l( \frac{1}{m}\r) - a(u_m)\r)\r| \\
 &\leq \l| \chi_{|u_m| \leq \frac{1}{m}} \l( -u_m\log\l( \frac{1}{m}\r)^2 +u_m\log |u_m|^2\r) \r| \\
 &\leq -\frac{2}{m}\log \l( \frac{1}{m}\r)^2 .
 \end{align*}
 Hence, $L_{11}$ is estimated as
 \begin{align}
 \label{eq:3.14}
 L_{11} &\leq 2\l( -\frac{2}{m}\log \l( \frac{1}{m}\r)^2 -\frac{2}{n}\log \l( \frac{1}{n}\r)^2\r) \| \psi_R^2 (u_m-u_n)\|_{L^1}\\
&\leq 8|\Omega_{2R} |^{\frac{1}{2}}\l( \frac{\log m}{m}+\frac{\log n}{n}\r) \| u_m-u_n \|_{L^2} \notag\\
&\leq 16|\Omega_{2R} |^{\frac{1}{2}}\l( \frac{\log m}{m}+\frac{\log n}{n}\r) \| \varphi \|_{L^2} ,
\notag 
\end{align}
where $\Omega_{2R}=\Omega \cap B(0 ,2R)$. By Lemma~\ref{lem:3.5}, $L_{12}$ is estimated as
\begin{align}
\label{eq:3.15}
L_{12} &= 2\im \rbra[ \psi_R^2 (a_n(u_m)-a_n(u_n) ), u_m-u_n]\\
&\leq C\| \psi_R (u_m-u_n)\|_{L^2}^2. \notag
\end{align}
Next, we estimate $L_2$. We write
\begin{align*}
L_2&=-2\im \rbra[\psi_R^2 (b_m(u_m)-b_n(u_m) ), u_m-u_n]\\
&\quad -2\im \rbra[\psi_R^2 (b_n(u_m)-b_n(u_n) ), u_m-u_n]\\
&=L_{21}+L_{22}.
\end{align*}
We note that
\begin{align*}
\l| b_m(u_m)-b(u_m)\r| &=\l|  \chi_{|u_m| \geq m} (b_m(u_m) -b(u_m) )\r| \\
&\leq C \chi_{|u_m| \geq m}|u_m| \log |u_m|^2,
\end{align*}
where $C$ is independent of $m$. Therefore, we have
\begin{align}
\label{eq:3.16}
\| b_m(u_m) -b(u_m) \|_{L^2}^2 &\leq C\int_{|u_m|\geq m}|u_m|^2\l( \log |u_m|^2\r)^2 dx\\
&\leq C\int_{|u_m|\geq m} |u_m|^{2(1+\eps )}|u_m|^{-\eps}dx \notag\\
&\leq C\| u_m\|_{H^1}^{2(1+\eps )}m^{-\eps}\leq C(M)m^{-\eps}, \notag
\end{align}
where we fix $\eps =\eps_N >0$ satisfying
\begin{align*}
2<2(1+\eps )<
\l\{
\begin{array}{ll}
\infty &~\text{if}~N=1,2,\\[3pt]
\frac{2N}{N-2} &~\text{if}~N\geq 3
\end{array}
\r.
\end{align*}
for each dimension $N$. Applying (\ref{eq:3.16}), $L_{21}$ is estimated as
\begin{align}
\label{eq:3.17}
L_{21}&\leq C(M)\l( m^{-\frac{\eps}{2}} +n^{-\frac{\eps}{2}} \r) \| u_m-u_n\|_{L^2}\\
&\leq C(M)\l( m^{-\frac{\eps}{2}} +n^{-\frac{\eps}{2}} \r) .
\notag
\end{align}
By Lemma~\ref{lem:3.5}, $L_{22}$ is estimated as
\begin{align}
\label{eq:3.18}
L_{22} &= 2\im \rbra[\psi_R^2 (b_n(u_m)-b_n(u_n) ), u_m-u_n]\\
&\leq C\| \psi_R (u_m-u_n)\|_{L^2}^2. \notag
\end{align}
By (\ref{eq:3.14}), (\ref{eq:3.15}), (\ref{eq:3.17}) and (\ref{eq:3.18}), we obtain that
\begin{align*}
\frac{d}{dt}\| \psi_R (u_m-u_n) \|_{L^2}^2 &\leq \frac{C(M)}{R} +16|\Omega_{2R} |^{\frac{1}{2}}\l( \frac{\log m}{m}+\frac{\log n}{n}\r) \| \varphi \|_{L^2}\\
&\quad +C(M)\l( m^{-\frac{\eps}{2}} +n^{-\frac{\eps}{2}} \r) +C\| \psi_R (u_m-u_n)\|_{L^2}^2 .
\end{align*}
Integrating over $[-T,T]$ for any $T>0$ and applying the Gronwall inequality, we deduce that
\begin{align*}
\| \psi_R (u_m-u_n)\|_{C([-T,T],L^2 )}^2 &\leq  e^{CT}C(M )T \biggl( \frac{1}{R} 
+|\Omega_{2R} |^{\frac{1}{2}}\l( \frac{\log m}{m}+\frac{\log n}{n}\r) \\
& \qquad +\l( m^{-\frac{\eps}{2}} +n^{-\frac{\eps}{2}} \r) \biggr) .
\end{align*}
Taking the $\limsup_{m,n\to\infty}$ in the preceding inequality, we have
\begin{align}
\label{eq:3.19}
\limsup_{m,n \to \infty} \| \psi_R (u_m-u_n)\|_{C([-T,T],L^2 )}^2 \leq e^{CT}\frac{C(M)T}{R}. 
\end{align}
Now fix any $R_0 >0$ and take $R>R_0$. By (\ref{eq:3.19}), we have 
\begin{align*}
\limsup_{m,n \to \infty} \| u_m-u_n\|_{C([-T,T],L^2(\Omega_{R_0}) )}^2 &\leq \limsup_{m,n \to \infty} \| \psi_R (u_m-u_n)\|_{C([-T,T],L^2(\Omega ) )}^2 \\
&\leq e^{CT}\frac{C(M)T}{R} \underset{R\to\infty}{\longrightarrow} 0.
\end{align*}
We deduce that $(u_m)_{m\in\N}$ is a Cauchy sequence in $C([-T,T] , L^2(\Omega_{R_0}) )$. Since $T>0$ and $R_0>0$ were arbitrary, combining with (\ref{eq:3.3}) and (\ref{eq:3.6}), there exists $u \in L^{\infty}(\R ,H^1_0(\Omega ))$ such that
\begin{align}
\label{eq:3.20}
 &u_m \to u ~\text{in}~C([-T,T] , L^2 (\Omega' ))~\text{for all}~T>0~\text{and}~\Omega' \subset\subset \Omega , \\
\label{eq:3.21}
&u_m(t) \wto u(t)~\text{in}~H^1_0 (\Omega )~\text{for all}~t\in \R.
\end{align}
By (\ref{eq:3.20}), (\ref{eq:3.7}) and Fatou's lemma, we deduce that
\begin{align}
\label{eq:3.22}
\| u\|_{L^{\infty} (\R ,X)} \leq C(\| \varphi \|_X ).
\end{align}
Therefore, $u \in L^{\infty} (\R, W)$. It is easily verified that $u$ satisfies 
\begin{align*}
i\del_t u +\Delta u +u\log (|u|^2) =0
\end{align*}
in the sense of distribution $\scD' (\R\times\Omega )$ and it follows $u \in W^{1, \infty}(\R , W^{\ast})$ from the equation and Lemma \ref{lem:3.2}. By (\ref{eq:3.3}), (\ref{eq:3.5}), (\ref{eq:3.20}), (\ref{eq:3.21}) and Fatou's lemma, we have
\begin{align}
\label{eq:3.23}
M(u(t)) \leq M(\varphi ) \quad\text{and}\quad E(u(t)) \leq E(\varphi )
\end{align}
for all $t\in \R$. 

Next, we prove uniqueness in the class $L^{\infty}(\R ,W)\cap W^{1,\infty}(\R,W^{\ast})$. Let $u$ and $v$ be two solutions of (\ref{LSE}). Taking the $W^{\ast}-W$ duality, we have
\begin{align*}
\frac{d}{dt}\| u-v\|_{L^2}^2 &=2\tbra[u_t-v_t,u-v]_{W^{\ast},W} \\
&=-\im\int_{\Omega} (u\log |u|^2-v\log |v|^2) (\overline{u}-\overline{v}).
\end{align*}
Applying  Lemma \ref{lem:3.4} and Gronwall's inequality, we obtain $u=v$. By uniqueness and (\ref{eq:3.23}), we deduce that
\begin{align}
\label{eq:3.24}
M(u(t)) = M(\varphi ) \quad\text{and}\quad E(u(t)) =E(\varphi )
\end{align}
for all $t \in \R$. By (\ref{eq:3.24}) and Lemma \ref{lem:3.1}, we obtain $u\in C(\R ,W)$. Finally, continuous dependence is proved by Lemma \ref{lem:3.1}, Lemma \ref{lem:3.4} and conservation of energy. We omit the detail.
\begin{remark}
Our approach is also applicable for the fractional logarithmic Schr\"{o}dinger equation studied in \cite{dSZ15, A17}.
\end{remark}
\section{The damped nonlinear Schr\"{o}dinger equation}
\label{sec:4}
In this section, we consider the following  damped nonlinear Schr\"{o}dinger equation:
\begin{align}
 \label{DSE}
 \l\{
\begin{array}{lll}
\ds i \del_{t}u + \Delta u +i \frac{u}{|u|^{\alpha}}= 0 &\text{in} &[0,\infty )\times\Omega , \\[3pt]
\ds u=0  &\text{on}  &[0,\infty ) \times \del \Omega ,\\[3pt]
u(0) =\varphi   &\text{on} &\Omega ,
\end{array}
\r.
\end{align}
where $0\leq \alpha \leq 1$ and $\Omega \subset \R^N$ is an open set. The damping term in the equation (\ref{DSE}) appears in mechanics (see \cite{AAC06, AD03} and references therein). The case $\alpha =0$ (linear damping) has been well studied; see e.g., \cite{T84,T90,F01,OT09,D12}. In the case $\alpha =1$, the damping term is referred to as {\it Coulomb friction}, and it has been studied in the case of wave equations in \cite{BCD07}. Carles and Gallo \cite{CG11} studied the Cauchy problem for (\ref{DSE}) on compact manifolds without boundary when $0<\alpha \leq1$ and proved that finite time extinction phenomena occurs when $N \leq 3$. The restriction of dimension comes from the exponents in Nash type inequalities. This result was extended in \cite{CO15} for more general nonlinearities. 
%

Our main aim of this section is to prove the following theorem.
\begin{theorem}
\label{thm:4.1}
 Let $N\geq 1$, $0<\alpha <1$ and assume that $\Omega$ is a bounded open set. For every $\varphi \in H^1_0 (\Omega )$, \textup{(\ref{DSE})} has a unique global solution $u \in C([0,\infty ) , H^1_0 (\Omega))\cap C^1([0,\infty ) , H^{-1} (\Omega)) $ of the equation \textup{(\ref{DSE})}. Moreover, $u$ satisfies the following a priori estimate:
 \begin{align*}
 \| u\|_{C([0,\infty ) , H^1_0  )} \leq \| \varphi \|_{H^1_0}.
 \end{align*}
 \end{theorem}
 \begin{remark}
 In \cite{CG11}, they constructed a global weak solution for (\ref{DSE}) on compact manifolds without boundary by compactness arguments. Our construction of strong solution can be applied for the case compact manifolds in the same way.
 \end{remark}
 \begin{remark}
 We can construct strong $H^2(\Omega )$ solutions in the similar way as the proof of Theorem \ref{thm:4.1} (see also \cite{CG11}).
 \end{remark}
For the proof of Theorem \ref{thm:4.1}, we consider the approximate equation of (\ref{DSE}). Given $m \in \N$. Set
\begin{align*}
g(u)&=\frac{u}{|u|^{\alpha}}, \\
g_m(u)&=\l\{
\begin{array}{ll}
g(u)  &\ds ~\text{if}~|u| \geq \frac{1}{m},\\[8pt]
mug\l( \frac{1}{m}\r)  & \ds ~\text{if}~|u| \leq \frac{1}{m},
\end{array}
\r.
\end{align*}
and consider the following problem:
\begin{align}
\l\{
\begin{array}{ll}
 i \partial_{t}u_m + \Delta u_m +ig_m(u_m)= 0, \\
 u_m(0)=\varphi .
\end{array}
\r.
 \label{eq:4.2}
 \end{align}
It is easily verified that there exists a unique solution $u_m \in C([0,\infty ), H^1_0 (\Omega ) )\cap C^1([0,\infty ) , H^{-1} (\Omega))$ of (\ref{eq:4.2}) and 
\begin{align}
 \label{eq:4.3}
 \| u_m(t) \|_{L^2} \leq \| \varphi \|_{L^2} \text{ and } \| \nabla u_m (t)\|_{L^2} \leq \| \nabla \varphi \|_{L^2}
 \end{align}
 for all $t\in \R$. 

Next, we prove $(u_m)_{m\in\N}$ is a Cauchy sequence in $C([0,T] , L^2(\Omega ) )$ for all $T>0$. By using the equation (\ref{eq:4.2}), we have
 \begin{align*}
 \frac{d}{dt}\| u_m-u_n\|_{L^2}^2 &= 2\mathrm{Im}(i\partial_t u_m-i\partial_t u_n, u_m-u_n)\\
 &=-2\mathrm{Im} \rbra[i(g_m(u_m)-g_n(u_n)), u_m-u_n] \\
 &=-2\mathrm{Re} \rbra[g_m(u_m)-g_n(u_m), u_m-u_n] \\
 &\quad -2\mathrm{Re} \rbra[g_n(u_m)-g_n(u_n), u_m-u_n] \\
 &=K_1+K_2 .
 \end{align*}
 We begin with the estimate of $K_1$. We note that 
 \begin{align*}
 |g_m (u_m) -g(u_m) |
 &=\l| \chi_{|u_m| \leq \frac{1}{m}} \l( m^{\alpha}u_m -\frac{u_m}{|u_m|^{\alpha}} \r) \r| \\
 &\leq \chi_{|u_m| \leq \frac{1}{m}} \l( m^{\alpha}\cdot \frac{1}{m} +|u_m|^{1-\alpha}\r) \\
 &\leq 2m^{-(1-\alpha )}.
 \end{align*}
 Therefore, $K_1$ is estimated as 
\begin{align*}
 K_1 &=-2\mathrm{Re} \rbra[g_m(u_m)-g(u_m), u_m-u_n]\\
&\quad -2\mathrm{Re} \rbra[g(u_m)-g_n(u_m), u_m-u_n] \\
&\leq 4\l( m^{-(1-\alpha )}+n^{-(1-\alpha )} \r) \| u_m -u_n\|_{L^1} \\
&\leq 8|\Omega |^{\frac{1}{2}} \| \varphi \|_{L^2} 
\l( m^{-(1-\alpha )}+n^{-(1-\alpha )} \r) .
 \end{align*}
 In order to estimate $K_2$, we use the following lemma. This result extends Lemma 2.6 in \cite{CO15}.
 \begin{lemma}
 \label{lem:4.2}
 Let $f:\R_{+} \to \R$ be a monotone increasing function with $f(0)=0$. We extend $f$ to the complex plane by setting $f(z)=\frac{z}{|z|} f(|z|)$ for $z \in \C\setminus\{ 0\}$. Then,
 \begin{align*}
 \mathrm{Re} \Bigl[ (f(z_1)-f(z_2))(\overline{z_1- z_2}) \Bigr] \geq 0
 \end{align*}
 for all $z_1, z_2 \in \C$.
 \end{lemma}
 \begin{proof}
We may $z_1, z_2 \neq 0$.  Then, we have
\begin{align*}
 &\quad \mathrm{Re} \l[ \l( \frac{z_1}{|z_1|}f(|z_1|)-\frac{z_2}{|z_2|}f(|z_2|)\r) (\overline{z_1- z_2}) \r] \\
 &=f(|z_1|)|z_1| - \frac{f(|z_1|)}{|z_1|} \mathrm{Re} (z_1\overline{z_2})- \frac{f(|z_2|)}{|z_2|} \mathrm{Re} (z_2 \overline{z_1})+f(|z_2|)|z_2| \\
 &\geq   f(|z_1|)|z_1| -  f(|z_1|)|z_2| - f(|z_2|)|z_1| + f(|z_2|) |z_2| \\
 &= \bigl( f(|z_1|) -f(|z_2|) \bigr) (|z_1| -|z_2|) \\
&\geq 0,
 \end{align*}
where in the last inequality we have used the monotonicity of $f: \R_{+} \to \R$.
 \end{proof}
 We note that $\R_{+} \ni s \to g_n(s)$ is  a monotone increasing function with $g_n(0)=0$. Applying $K_2$ to Lemma~\ref{lem:4.2}, we obtain that
 \begin{align*}
 K_2=-2\mathrm{Re} \rbra[g_n(u_m)-g_n(u_n), u_m-u_n] \leq 0.
 \end{align*}
 Hence, we obtain 
\begin{align*}
 \frac{d}{dt}\| u_m-u_n\|_{L^2}^2 \leq 8|\Omega |^{\frac{1}{2}} \| \varphi \|_{L^2} 
\l( m^{-(1-\alpha )}+n^{-(1-\alpha )} \r) .
\end{align*} 
Integrating over $[0, T]$ for any $T>0$, we obtain that
\begin{align*}
\| u_m-u_n\|_{C([0,T],L^2 )}^2 &\leq  8T|\Omega |^{\frac{1}{2}} \| \varphi \|_{L^2(\Omega )} 
\l( m^{-(1-\alpha )}+n^{-(1-\alpha )} \r) \\
&\underset{m,n\to\infty}{\longrightarrow} 0.
\end{align*}
Therefore, there exists $u \in C([0,\infty ) ,L^2(\Omega ))$ such that
\begin{align}
\label{eq:4.4}
u_n \to u ~\text{in}~ C([0,T],L^2(\Omega ) )~\text{for all}~T>0. 
\end{align}
From (\ref{eq:4.3}) and (\ref{eq:4.4}), we deduce that $u \in L^{\infty} ([0,\infty ) , H^1_0 (\Omega) )$ and 
\begin{align*}
\| u(t)\|_{H^1(\Omega )} \leq \| \varphi \|_{H^1(\Omega )}
\end{align*}
for all $t>0$.

We shall prove that the function $u$ satisfies (\ref{DSE}) and lies in $C([0,\infty ) , H^1_0(\Omega) )$. Note that $u_m$ is a solution of the integral equation
\begin{align}
u_m(t)=U(t)\varphi -\int_{0}^{t} U(t-s)g_m(u_m(s))ds \quad\text{for all}~t>0. \label{eq:4.5}
\end{align}
Note that
\begin{align*}
|g_m(u_m) -g(u)| &\leq |g_m(u_m) -g(u_m)| +|g(u_m) -g(u)| \\
&\leq 2 m^{-(1-\alpha )} +|g(u_m) -g(u)|.
\end{align*}
Using the following elementary inequality
\begin{align*}
\l| \frac{u}{|u|^{\alpha}}-\frac{v}{|v^{\alpha}|}\r| \lesssim |u-v|^{1-\alpha} \quad \text{for all}~ u ,v \in \C ,
\end{align*}
we deduce that
\begin{align*}
\| g(u_m(s)) -g(u(s)) \|_{L^2} \leq C| \Omega |^{\frac{\alpha}{2}} \| u_m(s)-u(s)\|^{1-\alpha}_{L^2}
\end{align*}
for all $s>0$. Hence, we obtain $g_m(u_m) \to g(u)$ in $C([0,T], L^2 (\Omega ))$ for all $T>0$. Taking the limit in the integral equation (\ref{eq:4.5}) as $m \to \infty$, we conclude that
\begin{align}
u(t) = U(t)\varphi -\int_{0}^{t} U(t-s)g(u(s)) ds \quad\text{for all}~t>0. \label{eq:4.6}
\end{align}
Since $g(u) \in C([0,\infty ), L^2(\Omega))$, 
it follows that $u \in C^1([0,\infty ), H^{-1}(\Omega))$ from (\ref{eq:4.6}). 
Therefore, $u$ satisfies the equation (\ref{DSE}). It follows from the equation (\ref{DSE}) that $\Delta u \in C([0,\infty ), H^{-1}(\Omega))$. Hence, we deduce that $u \in C([0,\infty ) , H^1_0(\Omega ))$. This completes the proof.
\begin{remark}
Carles and Gallo \cite{CG11} also studied the case $\alpha =1$. They introduced weak solutions in this case as follows. A global weak solution to (\ref{DSE}) is a function $u \in C([0,\infty) , L^2(\Omega ))\cap L^{\infty} ( (0 ,\infty) ,H^1_0 (\Omega) )$ solving
\begin{align*}
i\frac{\del u}{\del t} +\Delta u+iF =0
\end{align*}
in the sense of distribution $\scD' (\R_{+} \times\Omega )$, where $F$ is such that
\begin{align*}
\| F\|_{L^{\infty} (\R_{+}\times \Omega )} \leq 1,~\text{and}~F=\frac{u}{|u|}~\text{if}~u\neq 0.
\end{align*}
They proved that (\ref{DSE}) has a unique, global weak solution in this sense above by compactness arguments. We do not know at this stage that our approach covers the case $\alpha =1$.
In the proof above, the estimate of $K_2$ is done in the same way, while we can not obtain the decay of $m$ and $n$ from the estimate of $K_1$ when $\alpha =1$.
\end{remark}
\begin{remark}
Combining with the proof of Theorem \ref{thm:2.1}, by the same approach we can construct a global strong solutions in $H^1_0(\Omega )$ to 
\begin{align}
\label{eq:4.7}
 i \del_{t}u + \Delta u +|u|^{2\sigma}u+i \frac{u}{|u|^{\alpha}}= 0, 
\end{align}
where $0<\alpha <1$ and $\sigma < 2/N~(N=1, 2)$. The assumption of nonlinearity means that no finite time blow-up occurs without damping term. Finite time extinction phenomena of (\ref{eq:4.7}) on compact manifolds was investigated by Carles and Gallo \cite{CG11} when $N=1$, and Carles and Ozawa \cite{CO15} when $N=2$. Their proof still holds in the Dirichlet boundary problem.  It is an interesting open problem whether finite time blow-up or extinction occurs when $0<\alpha <1$ and $\sigma =2/N$. We refer to \cite{OT09,D12} for these problems with the linear damping case ($\alpha=0$).
\end{remark}
When $\Omega$ is unbounded, we have the following result.
\begin{theorem}
\label{thm:4.3}
 Let $N\geq 1$, $0<\alpha <1$ and assume that $\Omega$ is an unbounded open set. For every $\varphi \in H^1_0 (\Omega )$, \textup{(\ref{DSE})} has a unique global solution $u \in (C_w \cap L^{\infty})([0,\infty ) , H^1_0 (\Omega))$ of the equation \textup{(\ref{DSE})} in the sense of distribution $\scD' (\R_{+}\times\Omega )$. Moreover, $u$ satisfies the following a priori estimate:
 \begin{align*}
 \| u\|_{L^{\infty}((0,\infty ) , H^1_0 )} \leq \| \varphi \|_{H^1_0}.
 \end{align*}
 \end{theorem}
\begin{proof}
We can construct solutions if we modify the proof of Theorem \ref{thm:4.1} by using cut-off functions as in the proof of Theorem \ref{thm:3.3}. Here, we only prove uniqueness.

Assume that $u_1$ and $u_2$ are two solutions satisfying the equation (\ref{DSE}) in the sense of distribution $\scD' (\R_{+}\times\Omega )$. Note that $u_i$ ($i=1,2$) satisfies
\begin{align}
\label{eq:4.8}
i\del_t u_i +\Delta u_i +i\frac{u_i}{ |u_i|^{\alpha} } =0 \quad\text{in}~H^{-1} (\Omega' )
\end{align}
for all $\Omega' \subset\subset \Omega$. Let us fix a function $\psi \in C^{\infty}_c (\R^N)$ such that $\psi$ is radial and
\begin{align*}
\psi (x)=
\l\{
\begin{array}{ll}
1 & ~\text{if}~|x|\leq 1, \\[3pt]
0& ~\text{if}~|x|\geq 2,
\end{array}
\r.
\end{align*}
and set $\psi_R (x) =\psi (x/R)$. We set 
\begin{align*}
M= \max \l\{ \| u_1\|_{L^{\infty}((0,\infty ) , H^1_0 )} , \| u_2\|_{L^{\infty}((0,\infty ) , H^1_0  )}\r\} .
\end{align*}
By taking the $H^{-1}(\Omega_{2R} )-H^1_0(\Omega_{2R})$ duality and the equation (\ref{eq:4.8}), we have
\begin{align*}
 \frac{d}{dt}\| \psi_R (u_1-u_2)\|_{L^2}^2 &= 2\tbra[\del_t (\psi_R ( u_1-u_2) ) , \psi_R (u_1-u_2)]_{H^{-1}(\Omega_{2R}), H^1_0 (\Omega_{2R})}\\
 &=2\im \rbra[ \nabla (\psi_R^2 )\cdot\nabla (u_1-u_2),u_1-u_2]\\
&\quad -2\re \rbra[\psi_R^2 (g(u_1)-g(u_2) ), u_1-u_2] \\
&\leq \frac{C(M)}{R},
\end{align*}
where in the last inequality we have used Lemma \ref{lem:4.2}. Integrating over $[0,T]$ for any $T>0$, we obtain
\begin{align*}
\| \psi_R (u_1-u_2)\|_{C([0,T], L^2)} \leq \frac{C(M)T}{R}.
\end{align*}
Applying Fatou's lemma, we deduce that
\begin{align*}
\| u_1(t)-u_2(t)\|_{L^2} &\leq \liminf_{R\to\infty}\| \psi_R (u_1(t)-u_2(t) )\|_{L^2} \\
&\leq \liminf_{R\to\infty}\frac{C(M)T}{R}=0
\end{align*}
for all $t \in [0,T]$. This yields that $u_1=u_2$.
\end{proof} 
\section*{Acknowledgement}
The author would like to thank Tohru Ozawa for his advice and encouragement. This work was completed during the author's stay at Instituto de Matem\'{a}tica Pura e Aplicada (IMPA). The author is grateful to Felipe Linares for his hospitality. The author is also grateful to R\'{e}mi Carles for his numerous valuable comments to the first manuscript. This work was supported by Grant-in-Aid for JSPS Fellows 17J05828 and Top Global University Project, Waseda University.


\end{document}